\documentclass[10pt,oneside,english]{amsart}
\usepackage[T1]{fontenc}
\usepackage[latin9]{inputenc}
\usepackage[a4paper]{geometry}
\usepackage{amsthm}
\usepackage{amssymb}

\makeatletter
\numberwithin{equation}{section}
\numberwithin{figure}{section}
\theoremstyle{plain}
\newtheorem{thm}{\protect\theoremname}
  \theoremstyle{remark}
  \newtheorem{rem}[thm]{\protect\remarkname}
  \theoremstyle{plain}
  \newtheorem{cor}[thm]{\protect\corollaryname}
  \theoremstyle{plain}
  \newtheorem{lem}[thm]{\protect\lemmaname}

 \usepackage{xypic}
\theoremstyle{definition}

\makeatother

\usepackage{babel}
  \providecommand{\corollaryname}{Corollary}
  \providecommand{\lemmaname}{Lemma}
  \providecommand{\remarkname}{Remark}
\providecommand{\theoremname}{Theorem}

\begin{document}

\title{affine open subsets in $\mathbb{A}^{3}$ without the cancellation
property }

\author{Adrien Dubouloz}

\address{Cnrs, Institut de Math\'ematiques de Bourgogne, Universit\'e de
Bourgogne, 9 avenue Alain Savary - BP 47870, 21078 Dijon cedex, France}

\email{adrien.dubouloz@u-bourgogne.fr}


\begin{abstract}
We give families of examples of principal open subsets of the affine
space $\mathbb{A}^{3}$ which do not have the cancellation property.
We show as a by-product that the cylinders over Koras-Russell threefolds
of the first kind have a trivial Makar-Limanov invariant. 
\end{abstract}
\maketitle

\section*{Introduction }

The generalized Cancellation Problem asks if two algebraic varieties
$X$ and $Y$ with isomorphic cylinders $X\times\mathbb{A}^{1}$ and
$Y\times\mathbb{A}^{1}$ are isomorphic themselves. Although the answer
turns out to be affirmative for a large class of varieties including
the case when one of the varieties is the affine plane $\mathbb{A}^{2}$
\cite{Iitaka1977a,Miyanishi1980}, counter-examples exists for affine
varieties in any dimension$\geq2$, and the particular case when one
of the two varieties is an affine space $\mathbb{A}^{n}$, $n\geq3$,
still remains a widely open problem. 

The first counter-example for complex affine varieties has been constructed
by Danielewski \cite{Danielewski89} in 1989: he exploited the fact
that the non isomorphic affine surfaces $S_{1}=\left\{ xz=y^{2}-1\right\} $
and $S_{2}=\left\{ x^{2}z=y^{2}-1\right\} $ in $\mathbb{A}_{\mathbb{C}}^{3}$
can be equipped with free actions of the additive group $\mathbb{G}_{a}$
admitting geometric quotients in the form of non trivial $\mathbb{G}_{a}$-bundles
$\rho_{i}:S_{i}\rightarrow\tilde{\mathbb{A}}^{1}$, $i=1,2$ over
the affine line with a double origin. It then follows that the fiber
product $S_{1}\times_{\tilde{\mathbb{A}}^{1}}S_{2}$ inherits simultaneously
the structure of a $\mathbb{G}_{a}$-bundle over $S_{1}$ and $S_{2}$
via the first and the second projection respectively, but since $S_{1}$
and $S_{2}$ are both affine, the latter are both trivial, providing
isomorphisms $S_{1}\times\mathbb{A}^{1}\simeq S_{1}\times_{\tilde{\mathbb{A}}^{1}}S_{2}\simeq S_{2}\times\mathbb{A}^{1}$.
Since then, Danielewski's fiber product trick has been the source
of many new counter-examples in any dimension \cite{Fieseler1994,Dubouloz2007,Finston2008,Dubouloz2011},
some of these being very close to affine spaces either from an algebraic
or a topological point of view. 

However, a counter-example over the field of real numbers was constructed
earlier by Hochster \cite{Hochster1972} using the algebraic counterpart
of the classical fact from differential geometry that the tangent
bundle of the real sphere $S^{2}$ is non trivial but $1$-stably
trivial. His argument actually applies more generally to the situation
when a finitely generated domain $R$ over a field $k$ admits a non
trivial projective module $M$ of rank $n-1\geq1$ such that $M\oplus R\simeq R^{\oplus n}=R^{\oplus n-1}\oplus R$.
Indeed, these hypotheses immediately imply that the varieties $X={\rm Spec}_{R}({\rm Sym}\left(M\right))$
and $Y={\rm Spec}(R[x_{1},\ldots,x_{n}])$ are not isomorphic as schemes
over $Z={\rm Spec}(R)$ while their cylinders are. Of course, there
is no reason in general that $X$ and $Y$ are not isomorphic as $k$-varieties,
but this holds for instance when $Z$ does not admit any dominant
morphism from an affine space $\mathbb{A}_{k}^{n}$ since then any
isomorphism between $X$ and $Y$ necessarily descends to an automorphism
of $Z$ (\cite{Iitaka1977a,Drylo2007}). Recently, Jelonek \cite{Jelonek2009}
gave revival to Hochster idea by constructing families of examples
of non uniruled affine open subsets of affine spaces of any dimension
$\geq8$ with $1$-stably trivial but non trivial vector bundles,
which fail the cancellation property. 

While affine affine open subsets of affine spaces of dimension $\leq2$
always have the cancellation property (see e.g. \emph{loc.cit}), we
derive in this note from a variant of Danielewski's fiber product
trick that cancellation already fails for suitably chosen principal
open subsets of $\mathbb{A}^{3}$. 

As an application of our construction we also obtain that all cylinders
over Koras-Russell threefolds $X_{d,k,l}=\left\{ x^{d}z=y^{l}+x-t^{k}=0\right\} \subset\mathbb{A}^{4}$,
$d\geq2$ and $2\leq l<k$ relatively prime \cite{Kaliman1997}, have
a trivial Makar-Limanov invariant \cite{Makar-Limanov1996}.

\section{Principal open subsets in $\mathbb{A}^{3}$ without the cancellation
property }

For every $d\geq1$ and $l\geq2$, we denote by $B_{d,l}$ the surface
in $\mathbb{A}^{3}={\rm Spec}\left(\mathbb{C}\left[x,y,z\right]\right)$
defined by the equation $f_{d,l}=y^{l}+x-x^{d}z=0$ and by $U_{d,l}=\mathbb{A}^{3}\setminus B_{d,l}\simeq{\rm Spec}(\mathbb{C}\left[x,y,z\right]_{f_{d,l}})$
its open complement. By construction, $U_{d,l}$ comes equipped with
a flat isotrivial fibration $f_{d,l}\mid_{U_{d,l}}:U_{d,l}\rightarrow\mathbb{A}_{*}^{1}={\rm Spec}(\mathbb{C}\left[t^{\pm1}\right])$
with closed fibers isomorphic to the surface $S_{d,l}\subset\mathbb{A}^{3}={\rm Spec}\left(\mathbb{C}\left[X,Y,Z\right]\right)$
defined by the equation $X^{d}Z=Y^{l}+X-1$. A surface $S_{d,l}$
having no non constant invertible function, an isomorphism $\varphi:U_{d,l}\stackrel{\sim}{\rightarrow}U_{d',l'}$
necessarily maps closed fibers of $f_{d,l}$ isomorphically onto that
of $f_{d',l'}$. But since $S_{d,l}$ is isomorphic to $S_{d',l'}$
if and only if $\left(d',l'\right)=\left(d,l\right)$ (see e.g. \cite[Theorem 3.2 and Proposition 3.6]{Dubouloz2009}),
it follows that the threefolds $U_{d,l}$, $d\geq1$, $l\geq2$, are
pairwise non isomorphic. In contrast, we have the following result: 
\begin{thm}
\label{thm:Main} For every fixed $l\geq2$, the fourfolds $U_{d,l}\times\mathbb{A}^{1}$,
$d\geq1$, are all isomorphic. \end{thm}
\begin{proof}
We exploit the fact that every $U_{d,l}$ admits a free $\mathbb{G}_{a}$-action
defined by the locally nilpotent derivation $x^{d}\partial_{y}+ly^{l-1}\partial_{z}$
of its coordinate ring $\mathbb{C}\left[x,y,z\right]_{f_{d,l}}$.
A free $\mathbb{G}_{a}$-action being locally trivial in the \'etale
topology, it follows that a geometric quotient $\nu_{d,l}:U_{d,l}\rightarrow\mathfrak{S}_{d,l}=U_{d,l}/\mathbb{G}_{a}$
exists in the form of an \'etale locally trivial $\mathbb{G}_{a}$-bundle
over a certain algebraic space $\mathfrak{S}_{d,l}$. Then it is enough
to show that for every fixed $l\geq2$, the algebraic spaces $\mathfrak{S}_{d,l}$
are all isomorphic, say to a fixed algebraic space $\mathfrak{S}_{l}$.
Indeed, if so, then for every $d,d'\geq1$, the fiber product $U_{d,l}\times_{\mathfrak{S}_{l}}U_{d',l}$
will be simultaneously a $\mathbb{G}_{a}$-bundle over $U_{d,l}$
and $U_{d',l}$ via the first and second projection respectively whence
will be simultaneously isomorphic to the trivial $\mathbb{G}_{a}$-bundles
$U_{d,l}\times\mathbb{A}^{1}$ and $U_{d',l}\times\mathbb{A}^{1}$
as $U_{d,l}$ and $U_{d',l}$ are both affine. 

The algebraic spaces $\mathfrak{S}_{d,l}$ can be described explicitly
as follows: one checks that the isotrivial fibration $f_{d,l}:U_{d,l}\rightarrow\mathbb{A}_{*}^{1}$
becomes trivial on the Galois \'etale cover $\xi_{l}:\mathbb{A}_{*}^{1}={\rm Spec}(\mathbb{C}\left[u^{\pm1}\right])\rightarrow\mathbb{A}_{*}^{1}$,
$u\mapsto t=u^{l}$, with isomorphism $\Phi_{d,l}:S_{d,l}\times\mathbb{A}_{*}^{1}\stackrel{\sim}{\longrightarrow}U_{d,l}\times_{\mathbb{A}_{*}^{1}}\mathbb{A}_{*}^{1}$
given by $\left(X,Y,Z,u\right)\mapsto(u^{l}X,uY,u^{\left(1-d\right)l}Z,u)$.
The group $\mu_{l}$ of $l$-th roots of unity acts freely on $S_{d,l}\times\mathbb{A}_{*}^{1}$
by $\varepsilon\cdot\left(X,Y,Z,u\right)=\left(X,\varepsilon^{-1}Y,Z,\varepsilon u\right)$
and the $\mu_{l}$-invariant morphism $\pi_{d,l}={\rm pr}_{1}\circ\Phi_{d,l}:S_{d,l}\times\mathbb{A}_{*}^{1}\rightarrow U_{d,l}$
descends to an isomorphism $(S_{d,l}\times\mathbb{A}_{*}^{1})/\mu_{l}\simeq U_{d,l}$.
The $\mathbb{G}_{a}$-action on $U_{d,l}$ lifts via the proper \'etale
morphism $\pi_{d,l}$ to the free $\mathbb{G}_{a}$-action on $S_{d,l}\times\mathbb{A}_{*}^{1}$
commuting with that of $\mu_{l}$ defined by the locally nilpotent
derivation $u^{ld-1}(X^{d}\partial_{Y}+lY^{l-1}\partial_{Z})$ of
its coordinate ring $\mathbb{C}\left[X,Y,Z\right]/\left(X^{d}Z-Y^{l}-X+1\right)\left[u^{\pm1}\right]$.
The principal divisor $\left\{ X=0\right\} $ of $S_{d,l}\times\mathbb{A}_{*}^{1}$
is $\mathbb{G}_{a}$-invariant and it decomposes into the disjoint
union of irreducible divisors $D_{\eta}=\left\{ X=Y-\eta=0\right\} _{\eta\in\mu_{l}}\simeq{\rm Spec}(\mathbb{C}\left[Z\right][u^{\pm1}])$
on which $\mu_{l}$ acts by $D_{\eta}\ni(Z,u)\mapsto(Z,\varepsilon u)\in D_{\varepsilon\eta}$.
Now a similar argument as in \cite[Lemma 1.2]{Fieseler1994} implies
that for every $\eta\in\mu_{l}$, the $\mathbb{G}_{a}$-invariant
morphism ${\rm pr}_{X}\times{\rm id}:S_{d,l}\times\mathbb{A}_{*}^{1}\rightarrow\mathbb{A}^{1}\times\mathbb{A}_{*}^{1}$
restricts on $(S_{d,l}\times\mathbb{A}_{*}^{1})\setminus\bigcup_{\varepsilon\in\mu_{l}\setminus\left\{ \eta\right\} }D_{\varepsilon}$
to a trivial $\mathbb{G}_{a}$-bundle over $\mathbb{A}^{1}\times\mathbb{A}_{*}^{1}$.
Letting $C\left(l\right)$ be the scheme over $\mathbb{A}^{1}={\rm Spec}\left(\mathbb{C}\left[X\right]\right)$
obtained by gluing $l$ copies $C_{\eta}$, $\eta\in\mu_{l}$, of
$\mathbb{A}^{1}={\rm Spec}\left(\mathbb{C}\left[X\right]\right)$
outside their respective origins, it follows that ${\rm pr}_{X}\times{\rm id}$
factors through a $\mu_{l}$-equivariant $\mathbb{G}_{a}$-bundle
$\rho_{d,l}\times{\rm id}:S_{d,l}\times\mathbb{A}_{*}^{1}\rightarrow C\left(l\right)\times\mathbb{A}_{*}^{1}$,
where $\mu_{l}$ acts freely on $C\left(l\right)\times\mathbb{A}_{*}^{1}$
by $C_{\eta}\times\mathbb{A}_{*}^{1}\ni\left(X,u\right)\mapsto\left(X,\varepsilon u\right)\in C_{\varepsilon\eta}\times\mathbb{A}_{*}^{1}$. 

A quotient $(C(l)\times\mathbb{A}_{*}^{1})/\mu_{l}$ exist in the
category of algebraic spaces in the form of a principal $\mu_{l}$-bundle
$\sigma_{l}:C(l)\times\mathbb{A}_{*}^{1}\rightarrow\mathfrak{S}_{l}$,
and the above description implies that $\rho_{d,l}\times{\rm id}$
descends to an \'etale locally trivial $\mathbb{G}_{a}$-bundle $\tilde{\nu}_{d,l}:U_{d,l}\rightarrow\mathfrak{S}_{l}$
for which the diagram \[\xymatrix{S_{d,l}\times \mathbb{A}^1_* \ar[d]_-{\rho_{d,l}\times {\rm id}} \ar[r]^-{\pi_{d,l}} & U_{d,l}\simeq (S_{d,l}\times \mathbb{A}^1_*)/\mu_l \ar[d]^{\tilde{\nu}_{d,l}} \\ C(l)\times \mathbb{A}^1_* \ar[r]^{\sigma_l} & \mathfrak{S}_l,} \] is
cartesian. By virtue of the universal property of categorical quotients
one has necessarily $\mathfrak{S}_{d,l}\simeq\mathfrak{S}_{l}$ for
every $d\geq1$. In particular, the isomorphy type of $\mathfrak{S}_{d,l}$
depends only on $l$, which completes the proof. \end{proof}
\begin{rem}
The algebraic spaces $\mathfrak{S}_{l}=(C\left(l\right)\times\mathbb{A}_{*}^{1})/\mu_{l}$,
$l\geq2$, considered in the proof above cannot be schemes: indeed,
otherwise the image in $\mathfrak{S}_{l}$ of the point $\left(0,1\right)\in C_{1}\times\mathbb{A}_{*}^{1}\subset C\left(l\right)\times\mathbb{A}_{*}^{1}$
would have a Zariski open affine neighborhood $V$. But then the inverse
image of $V$ by the finite \'etale cover $\sigma_{l}:C\left(l\right)\times\mathbb{A}_{*}^{1}\rightarrow\mathfrak{S}_{l}$
would be a $\mu_{l}$-invariant affine open neighborhood of $\left(0,1\right)$
in $C\left(l\right)\times\mathbb{A}_{*}^{1}$, which is absurd since
$\left(0,1\right)$ does not even have a separated $\mu_{l}$-invariant
open neighborhood in $C\left(l\right)\times\mathbb{A}_{*}^{1}$. This
implies in turn that the free $\mathbb{G}_{a}$-action on $U_{d,l}$
defined by the locally nilpotent derivation $x^{d}\partial_{y}+ly^{l-1}\partial_{z}$
is not locally trivial in the Zariski topology. In contrast, the latter
property holds for its lift to $S_{d,l}\times\mathbb{A}_{*}^{1}$
via the \'etale Galois cover $\pi_{d,l}:S_{d,l}\times\mathbb{A}_{*}^{1}\rightarrow U_{d,l}$. 
\end{rem}
\noindent 
\begin{rem}
In Danielewski's construction for the surfaces $S_{i}=\{x^{i}z=y^{2}-1\}\subset\mathbb{A}^{3}$,
$i=1,2$, the geometric quotients $S_{i}/\mathbb{G}_{a}\simeq\tilde{\mathbb{A}}^{1}$,
$i=1,2$, were obtained from the categorical quotients $S_{i}/\!/\mathbb{G}_{a}={\rm Spec}(\mathbb{C}\left[x\right])$
taken in the category of affine schemes by replacing the origin by
two copies of itself, one for each orbit in the zero fiber of the
quotient morphism $q={\rm pr}_{x}:S_{i}\rightarrow\mathbb{A}^{1}$.
For the threefolds $U_{d,l}$, the difference between the quotients
$U_{d,l}/\!/\mathbb{G}_{a}={\rm Spec}(\mathbb{C}\left[x,y,z\right]_{f_{d,l}}^{\mathbb{G}_{a}})$
taken in the category of (affine) schemes and the geometric quotients
$\mathfrak{S}_{l}=U_{d,l}/\mathbb{G}_{a}$ is very similar : indeed,
we may identify $U_{d,l}$ with the closed subvariety of $\mathbb{A}^{3}\times\mathbb{A}_{*}^{1}={\rm Spec}(\mathbb{C}[x,y,z][t^{\pm1}])$
defined by the equation $x^{d}z=y^{l}+x-t$ in such a way that $f_{d,l}:U_{d,l}\rightarrow\mathbb{A}_{*}^{1}$
coincides with the projection ${\rm pr}_{t}\mid_{U_{d,l}}$. Then,
the kernel of the locally nilpotent derivation $x^{d}\partial_{y}+ly^{l-1}\partial_{z}$
of the coordinate ring of $U_{d,l}$ coincides with the subalgebra
$\mathbb{C}[x,t^{\pm1}]$ and so, the $\mathbb{G}_{a}$-invariant
morphism $q={\rm pr}_{x,t}:U_{d,l}\rightarrow\mathbb{A}^{1}\times L={\rm Spec}(\mathbb{C}[x][t^{\pm1}])$
is a categorical quotient in the category of affine schemes. One checks
easily that $q$ restricts to a trivial $\mathbb{G}_{a}$-bundle over
the principal open subset $\left\{ x\neq0\right\} $ of $\mathbb{A}^{1}\times\mathbb{A}_{*}^{1}$
whereas the inverse image of the punctured line $\left\{ x=0\right\} \simeq L$
is isomorphic to $\tilde{L}\times\mathbb{A}^{1}={\rm Spec}(\mathbb{C}[y,t^{\pm1}]/(y^{l}-t)[z])$
where $\mathbb{G}_{a}$ acts by translations on the second factor.
So we may interpret the geometric quotient $\mathfrak{S}_{l}=U_{d,l}/\mathbb{G}_{a}$
as being obtained from $U_{d,l}/\!/\mathbb{G}_{a}=\mathbb{A}^{1}\times L$
by replacing the punctured line $\left\{ x=0\right\} \simeq L$ not
by $l$ disjoint copies of itself but, instead, by the total space
$\tilde{L}$ of the nontrivial \'etale Galois cover ${\rm pr}_{t}:\tilde{L}\rightarrow L$.
\end{rem}
\vspace{0.2cm}

\indent The Koras-Russell threefolds $X_{d,k,l}$ are smooth complex
affine varieties defined by equations of the form $x^{d}z=y^{l}+x-t^{k}$,
where $d\geq2$ and $2\leq l<k$ are relatively prime.%
\footnote{These are called Koras-Russell threefolds of the first kind in \cite{Moser-Jauslin2010}.%
} While all diffeomorphic to the euclidean space $\mathbb{R}^{6}$,
none of these threefold is algebraically isomorphic to the affine
$\mathbb{A}^{3}$. Indeed, it was established by Kaliman and Makar-Limanov
\cite{Makar-Limanov1996,Kaliman1997} that they have fewer algebraic
$\mathbb{G}_{a}$-actions than the affine space $\mathbb{A}^{3}$
: the subring ${\rm ML}(X_{d,k,l})$ of their coordinate ring consisting
of regular functions invariant under all algebraic $\mathbb{G}_{a}$-actions
on $X_{d,k,l}$ is equal to the polynomial ring $\mathbb{C}\left[x\right]$,
while ${\rm ML}\left(\mathbb{A}^{3}\right)$ is \emph{trivial}, consisting
of constants only. However, it was observed by the author in \cite{DuboulozTG2009}
that the Makar-Limanov invariant ${\rm ML}$ fails to distinguish
the cylinder over the so-called Russell cubic $X_{2,2,3}$ from the
affine space $\mathbb{A}^{4}$. This phenomenon holds more generally
for cylinders over all Koras-Russell threefolds $X_{d,k,l}$: 
\begin{cor}
\label{cor:KR-threefolds} All the cylinders $X_{d,k,l}\times\mathbb{A}^{1}$
have a trivial Makar-Limanov invariant. \end{cor}
\begin{proof}
We consider $X_{d,k,l}\times\mathbb{A}^{1}$ as the subvariety of
${\rm Spec}\left(\mathbb{C}\left[x,y,z,t\right]\left[v\right]\right)$
defined by the equation $f_{d,l}-t^{k}=0$. Since ${\rm ML}(X_{d,k,l}\times\mathbb{A}^{1})\subset{\rm ML}(X_{d,k,l})=\mathbb{C}\left[x\right]$,
it is enough to construct a locally nilpotent derivation of $\mathbb{C}\left[x,y,z\right]\left[v\right]/(f_{d,l}-t^{k})$
which does not have $x$ in its kernel. One checks easily that ${\rm ML}(U_{1,l})=\mathbb{C}[f_{1,l}^{\pm1}]$
is the intersection of the kernels of the locally nilpotent derivations
$x\partial_{y}+ly^{l-1}\partial_{z}$ and $ly^{l-1}\partial_{x}+\left(z-1\right)\partial_{y}$
of $\mathbb{C}\left[x,y,z\right]_{f_{1,l}}$. Theorem \ref{thm:Main}
above implies in particular that ${\rm ML}(U_{d,l}\times\mathbb{A}^{1})\simeq{\rm ML}(U_{1,l}\times\mathbb{A}^{1})=\mathbb{C}[f_{1,l}^{\pm1}]$
and so, there exists a locally nilpotent derivation $\delta_{d,l}$
of $\Gamma(U_{d,l}\times\mathbb{A}^{1},\mathcal{O}_{U_{d,l}\times\mathbb{A}^{1}})=\mathbb{C}\left[x,y,z\right]_{f_{d,l}}\left[v\right]$
which does not have $x$ in its kernel. Up to multiplying it by a
suitable power of $f_{d,l}\in{\rm Ker}(\delta_{d,l})$, we may further
assume that $\delta_{d,l}$ is the extension to $\mathbb{C}\left[x,y,z\right]_{f_{d,l}}\left[v\right]$
of a locally nilpotent derivation of $\mathbb{C}\left[x,y,z\right]\left[v\right]$
which has $f_{d,l}$ but not $x$ in its kernel. This implies in particular
that $B_{d,l}\times\mathbb{A}^{1}={\rm Spec}\left(\mathbb{C}\left[x,y,z\right]/(f_{d,l})\left[v\right]\right)$
is invariant under the corresponding $\mathbb{G}_{a}$-action on $\mathbb{A}^{4}={\rm Spec}(\mathbb{C}\left[x,y,z\right]\left[v\right])$.
The projection $p={\rm pr}_{x,y,z,v}:X_{d,k,l}\times\mathbb{A}^{1}\rightarrow\mathbb{A}^{4}$
being a finite Galois cover with branch locus $B_{d,l}\times\mathbb{A}^{1}$,
it follows that the $\mathbb{G}_{a}$-action on $\mathbb{A}^{4}$
lifts to a one on $X_{d,k,l}\times\mathbb{A}^{1}$ for which $p:X_{d,k,l}\times\mathbb{A}^{1}\rightarrow\mathbb{A}^{4}$
is $\mathbb{G}_{a}$-equivariant. By construction, the corresponding
locally nilpotent derivation of $\mathbb{C}\left[x,y,z\right]\left[v\right]/(f_{d,l}-t^{k})$
does not have $x$ in its kernel. 
\end{proof}
\noindent In the proof above, we used the following classical fact
that we include here because of a lack of an appropriate reference. 
\begin{lem}
\label{lem:LiftingLemma} Let $X$ be a variety defined over a field
of characteristic zero and equipped with a non trivial $\mathbb{G}_{a}$-action,
let $Z$ be a normal variety and let $p:Z\rightarrow X$ be a finite
surjective morphism. Suppose that there exists a $\mathbb{G}_{a}$-invariant
affine open subvariety $U$ of $X$ over which $p$ restricts to an
\'etale morphism. Then there exists a unique $\mathbb{G}_{a}$-action
on $Z$ for which $p:Z\rightarrow X$ is a $\mathbb{G}_{a}$-equivariant
morphism.\end{lem}
\begin{proof}
The induced $\mathbb{G}_{a}$-action on the invariant affine open
subvariety $U$ of $X$ is determined by a locally nilpotent derivation
$\partial$ of $\Gamma(U,\mathcal{O}_{U})$. Since $p:p^{-1}\left(U\right)\rightarrow U$
is \'etale and proper, $p^{-1}\left(U\right)$ is an affine open
subvariety of $Z$ and $\partial$ lifts in a unique way to a derivation
of $\Gamma(p^{-1}(U),\mathcal{O}_{p^{-1}\left(U\right)})$ which is
again locally nilpotent by virtue of \cite{Vasconcelos1969}. By construction,
the latter defines a $\mathbb{G}_{a}$-action on $p^{-1}\left(U\right)$
for which the restriction of $p$ to $p^{-1}\left(U\right)$ is equivariant.
Now the assertion follows from  \cite[Lemma 6.1]{Seshadri1972} which
guarantees that the $\mathbb{G}_{a}$-action on $p^{-1}\left(U\right)$
can be uniquely extended to a one on $Z$ with the desired property. 
\end{proof}

\bibliographystyle{amsplain}

\end{document}